\documentclass[11pt,reqno]{amsart}
\usepackage[margin=1in]{geometry}

\usepackage{amssymb,amsmath,amsthm}
\usepackage{microtype}
\usepackage{graphicx}
\usepackage{subcaption}
\usepackage{fullpage}
\usepackage{scrextend}
\usepackage{siunitx}
\usepackage{mathtools}
\usepackage{bbm}
\usepackage{algorithm}
\usepackage{algpseudocode}
\usepackage{comment}

\usepackage{enumerate}
\usepackage{tikz,calc}
\usepackage{tikz-cd}
\usetikzlibrary{automata,positioning}
\usepackage{comment}
\usetikzlibrary{calc}
\usepackage{epsfig,graphicx}
\usepackage{listings}
\usepackage{enumitem}

\usepackage{float}
\usepackage{bbm}
\usepackage[colorlinks=true, pdfstartview=FitH, linkcolor=blue, citecolor=blue, urlcolor=blue]{hyperref}

\newcommand{\eps}{\varepsilon}
\newcommand{\defeq}{\vcentcolon=}

\newcommand{\Prob}{\mathbb P}
\newcommand{\E}{\mathbb E}
\newcommand{\Var}{\mathrm{Var}}
\newcommand{\norm}[1]{\left\|#1\right\|_\infty}

\renewcommand{\leq}{\leqslant}
\renewcommand{\le}{\leqslant}
\renewcommand{\geq}{\geqslant}
\renewcommand{\ge}{\geqslant}
\newcommand{\fe}{f_{\text{ext}}}

\newcommand{\N}{\mathbb{N}}
 
\newcommand{\Z}{\mathbb{Z}}

\newcommand{\Mod}[1]{\ (\text{mod}\ #1)}
\renewcommand{\Mod}[1]{{\ifmmode\text{\rm\ (mod~$#1$)}\else\discretionary{}{}{\hbox{ }}\rm(mod~$#1$)\fi}}

\newcommand{\allnotes}[1]{}
\renewcommand{\allnotes}[1]{\textit{#1}}

\newtheorem{theorem}{Theorem}[section]

\newtheorem{lemma}[theorem]{Lemma}

\theoremstyle{definition} 

\numberwithin{theorem}{section}

\allowdisplaybreaks

\begin{document}

\title{A note on the extensible no-three-in-line problem}

\author{Anubhab Ghosal}

\address{Mathematical Institute \\ University of Oxford \\ Oxford, UK OX2 6GG}
\email{ghosal@maths.ox.ac.uk}


\begin{abstract}
    We show the existence of a set $S\subset\mathbb{Z}^2$ avoiding collinear triples satisfying $|S\cap [n]^2|=\Omega(n/\sqrt{\log n})$ for sufficiently large $n$. This improves on the best-known lower bound on Erde's extensible no-three-in-line problem due to Nagy, Nagy and Woodroofe by $\sqrt{\log n}$, leaving the same gap to the trivial upper bound. Our construction is random.
\end{abstract}

\maketitle

\section{Introduction}

More than a century ago, Dudeney~\cite{Dud} posed the following problem -- What is the largest number of points $f(n)$ that you can place on the $n$ by $n$ grid so that no three lie on a line? The \emph{no-three-in-line problem} is now one of the most famous extremal questions in discrete geometry and naturally appears on several lists of open problems~\cite{Guy,BMP,Green}. An upper bound of $2n$ follows immediately by noticing that there can be at most $2$ points on each of the $n$ vertical lines of the grid. Interestingly, this is the best-known upper bound on $f(n)$ for any $n$. Perhaps also somewhat surprisingly at first glance, the upper bound is correct up to a constant factor -- there are several constructions known of no-three-in-line sets $S\subset [n]^2$ satisfying $|S|=\Omega(n)$. The earliest of these, due to Erd\H{o}s~\cite{roth1951problem}, uses a \emph{modular parabola}, that is the set of points
\[
\{(t,t^2\text{ mod } p):t \in [p]\}\subset [p]^2
\]
for some prime $p=(1-o(1))n$. As a quadratic curve in $\mathbb{F}_p^2$ intersects a linear curve in at most $2$ points, it follows that $f(n)\geq (1-o(1))n$. This has been improved by Hall, Jackson, Sudbery and Wild~\cite{hall1975some} to the current best-known $f(n)\geq (1.5-o(1))n$ by patching together several \emph{modular hyperbolae}.

Joshua Erde~\cite[Problem 5]{baber2023collection} raised the following related question: What is the fastest growing function $\fe(n)$ so that there is a no-three-in-line set $S\subset\mathbb{Z}^2$ satisfying
\[
|S\cap[n]^2|=\fe(n)\text{ for all sufficiently large }n?
\]
Crucially, the set $S\subset \mathbb{Z}^2$ is not allowed to depend on $n$, and so this is an \emph{extensible} version of the classical no-three-in-line problem. Note the immediate upper bound
\[
\fe(n)=O(n)
\]
follows from $f(n)=O(n)$. Erde~\cite{baber2023collection} conjectured that one cannot take $\fe(n)=\Omega(n)$. Green~\cite[Problem 72]{Green} reiterated this question and mentioned that Erde's conjecture is plausible ``due to the vague intuition that examples of large no-three-in-a-line subsets of $[-n,n]^2$ may have to come from constructions $\pmod p$ with the value of $p$ being somehow tied to the value of $n$", as we saw in the earlier examples. However, Nagy, Nagy and Woodroofe~\cite{nagy2023extensible} took the antithetical view, conjecturing that a greedy construction achieves $\fe(n)=\Omega(n)$ and provided computational evidence towards this. Further, they proved that one can take
\[
\fe(n)=\Omega\left(\frac{n}{\log^{1+\eps}(n)}\right)\text{ for any }\eps>0.
\]
They did so by taking several separated copies of Erd\H{o}s' construction along the curve $x/\log^{\eps}(x)$, and deleting any collinear triples formed. In this note, we show that one can take $\fe(n)=\Omega(n/\sqrt{\log(n)})$, thus improving the lower bound of Nagy, Nagy and Woodroofe by $\sqrt{\log(n)}$, and leaving the same gap to the upper bound.
\begin{theorem}\
\label{thm}
There exist constants $c>0$ and $n_0\in\N$ and a set $S\subset \Z^2$ with no three points of $S$ on a line satisfying
\[
|S\cap [n]^2|\geq \frac{cn}{\sqrt{\log n}} 
\]
for all $n\geq n_0$.
\end{theorem}
\noindent
Our construction is random: we select each point $x\in \mathbb{N}^2$ with probability approximately
\[
\frac{c}{\|x\|_\infty\sqrt{\log\|x\|_\infty}},
\]
for some small constant $c>0$, independently for each $x$, to obtain the set $Q$. We then delete the point with the largest infinity norm from each collinear triple in $Q$ to obtain the set $S$. We show that the set $S$ so obtained satisfies the conclusion of Theorem \ref{thm} with positive probability for some small enough absolute constant $c>0$.

It is plausible that one can improve on Theorem \ref{thm} by a polynomial in $\log\log n$ by using random processes -- see the discussion in ~\cite[Section 1.1]{ghosal2025subsets}. Improving beyond that would likely require different or significantly new ideas.

In a forthcoming paper, the author and Goenka~\cite{GG} use similar methods to solve the extensible no-four-on-a-circle problem up to a constant factor. 

\section*{Acknowledgments}

The author thanks Ritesh Goenka for many helpful conversations and Rob Morris for his comments, which have greatly improved the presentation of this paper. The author is grateful to the Clarendon Fund and Oxford Ryniker Lloyd Graduate Scholarship for supporting his research.

\begin{section}{Proof of Theorem \ref{thm}}
\label{sec:proof}

For $T\in \N\cup\{0\}$, define respectively  the dyadic shell and the box as
\[
R_T\defeq\{x\in \N^2:2^T\le \norm{x}<2^{T+1}\}\quad \text{and} \quad B_T\defeq[2^T]^2.
\]
Fix a constant $c>0$, to be chosen sufficiently small later.
For each $T\geq 0$ and $x\in R_T$, include $x$ independently in a random set $Q\subset \mathbb{Z}^2$ with probability
\[
p(x):=\frac{c}{2^T\sqrt T}.
\]

For $T\geq 0$, define
\[
X_T\defeq |Q\cap R_T|.\]
Note that $\E[X_T]=|R_T|\frac{c}{2^T\sqrt T}\geq \frac{c2^T}{\sqrt T}$. As $X_T$ is a binomial random variable, $\Var(X_T)\leq \E[X_T]$. Therefore, by Chebyshev's inequality, we have that
\begin{align}
\label{eq: xprob}
    \Prob\left(X_T\leq \frac{c\,2^{T-1}}{\sqrt T}\right)\leq \Prob\left(|X_T-\E(X_T)|\geq \frac{1}{2}\E(X_T)\right)\leq\frac{4}{\E(X_T)}\leq \frac{4\sqrt{T}}{c2^T}.
\end{align}

Let $Y_T$ denote the number of unordered collinear triples contained in $Q\cap B_T$. We claim that the following bounds hold on the expectation and variance of $Y_T$.

\begin{lemma}
\label{claim2}
    There exists an absolute constant $K_1>0$ so that
    \[
    \E[Y_T]\leq K_1\frac{c^3\,2^T}{\sqrt T}
    \]
    holds for every $T\in \mathbb{N}\cup \{0\}$.
\end{lemma}

\begin{lemma}
\label{claim3}
    There exists an absolute constant $K_2>0$ so that 
    \[
    \Var(Y_T)\leq K_2 2^TT^{7/2}
    \]
    holds for every $T\in \mathbb{N}\cup \{0\}$.
\end{lemma}
\noindent We defer the proofs of the lemmas for now and see how to deduce Theorem \ref{thm} from them first.

\begin{proof}[Proof of Theorem \ref{thm}]
From Lemmas \ref{claim2} and \ref{claim3}, using Chebyshev's inequality, we get that
\begin{align}
\label{eq: yprob}
    \Prob\left(Y_T\geq 2K_1\frac{c^3\,2^T}{\sqrt T}\right)\leq \frac{T}{K_1^2c^62^{2T}}\cdot\Var(Y_T)\leq \frac{K_2}{K_1^2c^6}\cdot\frac{T^{9/2}}{2^{T}}.
\end{align}

We now choose $c\defeq 1/\sqrt{16K_1}$. For $T\geq 0$, define the event
\[
E_T=\left\{X_T\geq \frac{c\,2^{T-1}}{\sqrt T}\right\}\cap\left\{Y_T\leq 2K_1\frac{c^3\,2^T}{\sqrt T}\right\}.
\]
It follows from \eqref{eq: xprob} and \eqref{eq: yprob} that $\sum_{T\geq 0}\Prob(E_T^c)<\infty$. In particular, there exists $T_0\geq 0$ so that $\sum_{T\geq T_0}\Prob(E_T^c)<1$ and so, with positive probability, $E_T$ holds for all $T\geq T_0$. We pick a realisation $Q\subset\mathbb{Z}^2$ satisfying this property.

From this set $Q\subset \mathbb{Z}^2$, construct the set $S\subset Q$ as follows:
for each collinear triple in $Q$, delete from that triple the point of largest $\norm{\cdot}$, breaking any ties arbitrarily. By its construction, $S$ has no three points on a line. Deleting the point with the largest infinity norm ensures that $|S\cap R_T|\geq X_T-Y_{T+1}$. Therefore, for $n> 2^{T_0+2}$, letting $T\defeq \lfloor\log_2(n)\rfloor-1$, we have that
\[
|S\cap [n]^2|\geq |S\cap R_{T}|\geq X_T-Y_{T+1}\geq \frac{c}{4} \frac{2^T}{\sqrt{T}}\geq \frac{c}{16}\cdot\frac{n}{\sqrt{\log(n)}},
\]
as we wished.
\end{proof}

Before we can proceed with the proofs of the lemmas, we need to set up some notation. A \emph{lattice line} is a line in $\mathbb{R}^2$ intersected with the lattice $\mathbb{Z}^2$, provided the intersection contains at least two points. A nonzero vector $v=(a,b)\in\Z^2$ is called \emph{primitive} if $\gcd(a,b)=1$. Each lattice line is parallel to a primitive vector, unique up to sign: we fix one representative for each unoriented primitive direction, and call this vector the \emph{direction} of the line. For a line $L$ and a vector $v$, we write $L\parallel v$ to mean that the direction of $L$ is $v$. Finally, for functions $f,g:\mathbb{N}\to \mathbb{R}$, we write $f\ll g$ to mean $f=O(g)$, that is, there is a constant $C>0$ so that $|f(n)|\leq C|g(n)|$ for all $n\in \mathbb{N}$. We are now ready to prove Lemmas \ref{claim2} and \ref{claim3}.

\begin{proof}[Proof of Lemma \ref{claim2}]
    For a lattice line $L$, define its weight $W_T(L):=\sum_{x\in L\cap B_T} p(x)$. The expected number of triples on a line $L\cap Q$ is at most
    \[
    \sum_{x,y,z\in L\cap B_T} p(x)p(y)p(z)\le W_T(L)^3.
    \]
    It follows that
    \begin{equation}
    \label{eq: expectation is sum weight cubed}
        \E Y_T\le \sum_L W_T(L)^3
    \end{equation}
    where the sum runs over all lattice lines meeting $B_T$. Let $L$ be a lattice line with primitive direction $v$, let $m=\norm{v}$, and let $s(L):=\min\{t\ge 0:L\cap R_t\neq \emptyset\}$. Note that as
    \[
    |L\cap R_t|\ll\frac{2^t}{m}
    \]
    for $t\geq 0$, it follows that
    \begin{equation}
    \label{eq:weight on line bd}
        W_T(L)
    =\sum_{t=s(L)}^T |L\cap R_t|\frac{c}{2^t\sqrt t}
    \ll\frac{c}{m}\sum_{t=s(L)}^T \frac1{\sqrt t}\ll\frac{c}{m}(\sqrt{T}-\sqrt{s(L)-1}),
    \end{equation}
    where the last inequality follows from comparison with the corresponding integral.
    Now, we will bound the contribution to $\E[Y_T]$ from lines of direction $v=(a,b)$. Such a lattice line has equation $bx-ay=k$ for some $k\in\Z$. If $(x,y)\in R_t$, then
    \[
    |k|\le |b|2^{t+1}+|a|2^{t+1}\le 4m2^t,
    \]
    so the number of such lines meeting $R_t$ is $O(2^tm)$. It follows from \eqref{eq:weight on line bd} that
    \begin{align}
    \label{eq: total weight cubed on parallel lines}
        \sum_{L\parallel v} W_T(L)^3
    \ll \sum_{t=1}^T (2^t m)\left(\frac{c}{m}(\sqrt T-\sqrt{t-1})\right)^3
    \ll \frac{c^3}{m^2}\sum_{t=1}^T 2^t\left( \frac{T-{t+1}}{\sqrt{T}}\right)^3\ll\frac{c^32^T}{m^2 T^{3/2}},
    \end{align}
    where we used in the final step that $\sum_{r\geq 0}2^{-r}(r+1)^3$ converges.
    The number of primitive directions $v$ with $\norm{v}=m$ is $O(m)$, so summing over $m$ gives
    \[
    \sum_L W_T(L)^3
    \ll \frac{c^3\,2^T}{T^{3/2}}\sum_{m=1}^{2^T}\frac1m
    \ll \frac{c^3\,2^T}{\sqrt T}.
    \]
    The claim now follows from \eqref{eq: expectation is sum weight cubed}.
\end{proof}

\begin{proof}[Proof of Lemma \ref{claim3}]
    Note that $Y_T=\sum_{\tau} I_\tau$, where $\tau$ ranges over all unordered collinear triples in $B_T$, and $I_\tau$ is the indicator that all three points of $\tau$ are selected. It follows that
    \begin{align}
    \label{eq: second moment}
        \Var(Y_T)=\sum_{\tau,\tau'} \operatorname{Cov}(I_\tau,I_{\tau'})=\sum_{\tau \cap \tau'\neq \emptyset} \operatorname{Cov}(I_\tau,I_{\tau'})\leq \sum_{\tau \cap \tau'\neq \emptyset} \E(I_\tau I_{\tau'}).
    \end{align}
    For $1\leq i\leq 3$, let
    \[
    V_i\defeq \sum_{|\tau \cap \tau'|=i} \E(I_\tau I_{\tau'})
    \]
    First, note that
    \begin{equation}
\label{eq: V_3 bd}
    V_3\leq \E(Y_T)\ll2^T/\sqrt{T}.
    \end{equation}

    Next, we bound the total contribution from the terms satisfying $|\tau\cap \tau'|=2$.
    Such pairs lie on a common lattice line and together involve exactly four distinct points. For a fixed line $L$, the total contribution from all such ordered pairs is
    at most $W_T(L)^4$, where $W_T(L):=\sum_{x\in L\cap B_T} p(x)$ is defined as in the proof of Lemma \ref{claim2}. Suppose $v$ is a primitive vector satisfying $\norm{v}=m$. Similarly as in the derivation for \eqref{eq: total weight cubed on parallel lines}, we obtain that 
    \begin{align}
    \label{eq: total weight fourth on parallel lines}
        \sum_{L\parallel v} W_T(L)^4
    \ll \sum_{t=1}^T (2^t m)\left(\frac{c}{m}(\sqrt T-\sqrt{t-1})\right)^4
    \ll \frac{c^4}{m^3}\sum_{t=1}^T 2^t\left( \frac{T-{t+1}}{\sqrt{T}}\right)^4\ll\frac{c^42^T}{m^3 T^{2}},
    \end{align}
    Summing over primitive directions with $\norm{v}=m$, of which there are $O(m)$, yields
    \begin{equation}
\label{eq: V_2 bd}
    V_2\leq \sum_L W_T(L)^4
    \ll \frac{c^4\,2^T}{T^2}\sum_{m=1}^{2^T}\frac1{m^2}
    \ll \frac{\,2^T}{T^2}.
    \end{equation}

    Finally, we need to bound the total contribution of the terms satisfying $|\tau\cap \tau'|=1$. Note that the total contribution from ordered pairs of triples sharing exactly the point $x\in B_T$ is at most $p(x)\,\beta_T(x)^2$, where \[
\beta_T(x):=
\sum_{\{y,z\}} p(y)p(z),
\]
and the sum runs over unordered pairs $\{y,z\}$ of distinct points in $B_T\setminus\{x\}$ such that $x,y,z$ are collinear. It follows that
\begin{equation}
\label{eq: V_1 bdd by betas}
    V_1\le \sum_{x\in B_T} p(x)\,\beta_T(x)^2.
\end{equation}
For a line $L$ containing $x$, define $W_{T,x}(L):=\sum_{y\in L\cap B_T\setminus\{x\}} p(y)$.
Then
\begin{equation}
\label{eq: beta bounded by ws}
    \beta_T(x)\le \sum_{L\ni x} W_{T,x}(L)^2.
\end{equation}
If $L$ has primitive direction $v$ with $m=\norm{v}$, then similarly as in the proof of Lemma \ref{claim2}, we have that
\[
W_{T,x}(L)\ll \sum_{t=1}^T |L\cap R_t|\frac{c}{2^t\sqrt t}
\ll \frac{c}{m}\sum_{t=1}^T \frac1{\sqrt t}
\ll \frac{c\sqrt T}{m}.
\]
The number of primitive directions with $\norm{v}=m$ is $O(m)$, and for each such direction there is exactly one line through $x$.
Hence, using \eqref{eq: beta bounded by ws}, we have that
\[
\beta_T(x)\leq \sum_{L\ni x} W_{T,x}(L)^2
\ll \sum_{m=1}^{2^T} m\left(\frac{c\sqrt T}{m}\right)^2
= c^2 T \sum_{m=1}^{2^T}\frac1m
\ll c^2 T^2.
\]
Now, \eqref{eq: V_1 bdd by betas} yields
\begin{equation}
\label{eq: V_1 bd}
    V_1\ll c^4T^4\sum_{x\in B_T}p(x)=c^4T^4\sum_{t\leq T}\E[X_t]\ll T^{7/2}2^T,
\end{equation}
where in the last line, we used that $\E[X_t]=\Theta(2^t/\sqrt{t})$. Combining \eqref{eq: second moment}, \eqref{eq: V_3 bd}, \eqref{eq: V_2 bd}, and \eqref{eq: V_1 bd}, we have that
\[
\Var(Y_T)\leq V_3+V_2+V_1\ll 2^TT^{7/2},
\]
as desired.
\end{proof}

\end{section}

\bibliographystyle{amsplain}
\bibliography{references}

\end{document}